\documentclass{article}

\usepackage{amsmath, amssymb, amsthm, mathtools, setspace}
\usepackage{xcolor}
\usepackage{epsfig}
\usepackage{mathtools}
\usepackage[subnum]{cases}
\usepackage{enumerate}
\usepackage{bbm,bm}
\usepackage{amsmath,amssymb,mathrsfs,amsthm}
\usepackage{color}
\usepackage{hyperref}
\usepackage{verbatim}
\usepackage{tikz}
\usepackage{multicol}
\usepackage{multirow}
\usepackage{wrapfig}
\hypersetup{
    colorlinks,
    citecolor=black,
    filecolor=black,
    linkcolor=black,
    urlcolor=black
}

\usepackage{subfig}
\usepackage{listings}
\definecolor{gris245}{RGB}{245,245,245}
\definecolor{olive}{RGB}{50,140,50}
\definecolor{brun}{RGB}{175,100,80}

\newtheorem{defn}{Definition}
\newtheorem*{defn*}{Definition}
\newtheorem{theorem}{Theorem}
\newtheorem*{theorem*}{Theorem}
\newtheorem{cor}{Corollary}
\newtheorem*{cor*}{Corollary}

\newtheorem*{prop*}{Proposition}
\newtheorem{lemma}{Lemma}
\newtheorem*{lemma*}{Lemma}
\newtheorem{remark}{Remark}
\newtheorem*{remark*}{Remark}

\newtheorem*{property*}{Property}

\newtheorem*{problem*}{Problem}

\theoremstyle{remark}
\newtheorem*{sketch*}{Sketch of Proof}

\title{A note concerning fundamental functions of interpolation associated to the inverse multiquadric $(\alpha^2+x^2)^{-k}$}
\author{Jeff Ledford, Kyle Rutherford}
\date{}

\begin{document}
\maketitle 

\section{Introduction}

This short note develops fundamental functions associated with the scattered shifts of the inverse \emph{multiquadric} function $(\alpha^2 + x^2)^{-k}$, for $k\in\mathbb{N}$.  Specifically, we solve the following interpolation problem.  
\begin{problem*}
    Let $k\in\mathbb{N}$ and suppose that $X=(x_j:j\in\mathbb{Z})$ is a complete interpolating sequence and $y=(y_j:j\in\mathbb{Z})\in\ell^2$.  For each $m\in\mathbb{Z}$, find coefficients $a=(a_j:j\in\mathbb{Z})\in\ell^2$ such that
    \[
    L_m(x)=\sum_{j\in\mathbb{Z}} a_{j}\left(\alpha^2+(x-x_j)^2\right)^{-k}
    \]
    satisfies
    \[
    L_m(x_k)=\begin{cases}
        1 & k=m;\\
        0 & k\neq m.
    \end{cases}
    \]
\end{problem*}

This problem has been investigated before, mainly in the \emph{cardinal} case, where $X=\mathbb{Z}$.  One may consult \cite{HL_2016, HL_2018, Ledford_2015} for various properties corresponding to the cardinal interpolation problem.  We also find this example as a `regular interpolator' in \cite{Ledford_2013}, which solves the above problem using Fourier analysis.  Notably, \cite{HL_2016} show that the fundamental function for cardinal interpolation obeys a $O(|x|^{-2k+2})$ decay rate.  

However, we follow the template set out in \cite{SS_2009}, which develops fundamental functions for scattered shifts of Gaussians.  In particular, decay properties are derived without Fourier analysis. Some of our results may be found in a more general form in \cite{Jaffard_1990}.  We include limited versions with different proofs.  These techniques allow us to improve the decay rate of fundamental function to $O(|x|^{-2k})$.

The remainder of the paper is laid out as follows.  Section 2 contains various definitions and basic facts necessary for the sequel, while Section 3 contains the main results.


\section{Definitions and Basic Facts}

We use the symbols $\mathbb{R},\mathbb{Z},$ and $\mathbb{N}$ to denote the set of real numbers, integers, and natural numbers, respectively.  In the sequel, we assume the reader is familiar with calculus concepts; with this in mind, we make special note of the following facts.

\section{Results}

\noindent Throughout this section, we fix $\alpha>0$, $k\in\mathbb{N}$, and a CIS $X=(x_j:j\in\mathbb{Z})$.
Our interpolation problem hinges on the the analysis of the (bi-infinite) matrix operator $A:\ell^2(\mathbb{Z})\to\ell^2(\mathbb{Z})$ whose entries are given by 
\[
A(i,j) = (\alpha^2+(x_i-x_j)^2)^{-k}.
\]
The invertibility of this operator follows from Lemma 2 in \cite{Ledford_2013}.  This allows us to follow \cite{SS_2009}, Section 6 to develop fundamental functions.  The exponential decay of the Gaussian aids in their computations, a feature which we will be unable to exploit.  We begin by establishing a growth rate for the entries of the inverse.

\begin{theorem}
    Let $A=[A(i,j)]$ be a bi-infinite matrix operator on $\ell^2$ that is self-adjoint, positive, and invertible.  Suppose that there exist constants $C>0$ and $k>1$ such that ${|A(i,j)|\leq C |i-j|^{-k}}$ for all $i,j\in\mathbb{Z}$.  Then there exist constants $\tilde{C}>0$ and $\tilde{k}>1$ such that ${|A^{-1}(m,n)|\leq \tilde{C}|m-n|^{-\tilde{k}}}$ for all $m,n\in\mathbb{Z}$ .
\end{theorem}
In order to prove this we will use Lemma 6.2 in \cite{SS_2009} together with the following result.
\begin{lemma}
Suppose that $(R(s,t))_s,_t\in\mathbb{Z}$ is a bi-infinite matrix satisfying the following condition: there exist positive constants $C$ and $k\geq 1$ such that $|R(s,t)|\leq {C(\alpha^2+(s-t)^2)^{-k}}$ for every pair of integers $s$ and $t$. There is a constant $C_{\alpha,k}>0$, such that $|R^n(s,t)| \leq {C^n C_{\alpha,k}^{n-1}(\alpha^2+(s-t)^2)^{-k}}$ for every $s, t \in\mathbb{Z}$.
\end{lemma}


  \begin{proof}
      
  We begin by letting $p\in(1/2,k]$ and suppose that $s \neq t \in\mathbb{Z}$, and assume without loss that $s < t$.  We have
  \begin{align*}
      &\sum_{u=-\infty}^\infty(\alpha^2+(s-u)^2)^{-k}(\alpha^2+(t-u)^2)^{-p} \\
      =& \sum_{u=s}^t(\alpha^2+(s-u)^2)^{-k}(\alpha^2+(t-u)^2)^{-p} \\
      &\quad + \sum_{u=-\infty}^{s-1}(\alpha^2+(s-u)^2)^{-k}(\alpha^2+(t-u)^2)^{-p} \\
      & \qquad+  \sum_{u=t+1}^\infty(\alpha^2+(s-u)^2)^{-k}(\alpha^2+(t-u)^2)^{-p} \\
      =:& \sum_1 + \sum_2 + \sum_3.
  \end{align*}
  
  Now, letting $j = u-s $, we have
 \begin{align}
     \nonumber\sum_1  =& 
     \sum_{j=0}^{t-s}(\alpha^2+j^2)^{-k}(\alpha^2+(t-(s+j))^2)^{-p} \\
  \label{sum1}   =& \sum_{j=0}^{t-s}(\alpha^2+j^2)^{-k}(\alpha^2+(t-s-j)^2)^{-p} 
 \end{align}

Setting $M=t-s$, reindexing, then applying partial fractions yields
\begin{align*}
    \sum_{1} & = \sum_{j=0}^M(\alpha^2+j^2)^{-k}(\alpha^2+(M-j)^2)^{-p} \\
    & \leq \alpha^{-2(k-p)}\sum_{j=0}^{M} \left[\frac{1}{(\alpha^2+j^2)(\alpha^2+(M-j)^2}\right]^p \\
    & =\alpha^{-2(k-p)}\sum_{j=0}^{M}(4\alpha^2+M^2)^{-p}\left[ \frac{(\frac{2}{M})j+1}{\alpha^2+j^2}+\frac{3-(\frac{2}{M})j}{\alpha^2+(M-j)^2} \right]^{p}\\
    & \leq C_{\alpha,p}(\alpha^2+M^2)^{-p}.
\end{align*}

The bounds for $\sum_2$ and $\sum_3$ work similarly, reindexing, we have
 \begin{align*}
     \sum_2 
     &= \sum_{j=1}^\infty(\alpha^2+j^2)^{-k}(\alpha^2+(t-s+j)^2)^{-p}\\
     &\leq (\alpha^2+(t-s)^2)^{-p}\sum_{j=1}^{\infty}(\alpha^2+j^2)^{-k} \\
     &=C_{\alpha,k}(\alpha^2+(t-s)^2)^{-p}
 \end{align*}
 and
 \begin{align*}
     \sum_3 
     &=\sum_{j=1}^\infty(\alpha^2+(s-j-t)^2)^{-k}(\alpha^2+j^2)^{-p} \\
     &\leq (\alpha^2+(t-s)^2)^{-k}\sum_{j=1}^{\infty}(\alpha^2+j^2)^{-p} \\
     &=C_{\alpha,p}(\alpha^2+(t-s)^2)^{-k}
 \end{align*}
 
 If $s=t$, since $p\leq k$, we have
 \[
     \sum_{u=-\infty}^{\infty}(\alpha^2+(s-u)^2)^{-k-p}= \sum_{u=-\infty}^{\infty} (\alpha^2+j^2)^{-k-p}
     \leq C_{\alpha,k}.
 \]
  Setting $p=k$ and combining the estimates above, we conclude that
 \begin{align*}
     |R^2(s,t)|
     &= \sum_{u=-\gamma}^{\infty} (\alpha^2+(s-u)^2)^{-k} (\alpha^2+(t-u)^2)^{-k} \\
     &\leq C_{\alpha, k}(\alpha^2+(t-s)^2)^{-k}.
 \end{align*}
 The theorem follows by induction.
 
 \end{proof}

 \begin{proof}[Proof of Theorem 1]
    We may expand $A^{-1}$ via its Neumann series.
    \begin{align*}
        A^{-1} &= \Vert A \Vert^{-1} \sum_{j=0}^{N-1} R^j + \Vert A \Vert^{-1} R^N \sum_{j=0}^\infty R^j \\
        &= \Vert A \Vert^{-1} \sum_{j=0}^{N-1} R^j + R^N A^{-1}.
        \end{align*}
Thus $|A^{-1}(s,t)| \leq \Vert A \Vert^{-1} \sum_{j=0}^{N-1} R^j(s,t) + r^N \Vert A^{-1} \Vert $, where $r:=\| R\|\in(0,1)$ by \cite[Lemma 6.2]{SS_2009}.
Now we choose $\alpha$ so large that $C_{\alpha,k}$ from \cite[Lemma  6.3]{SS_2009} satisfies $C_{\alpha, k} < 1$. This yields
\[
\frac{D^N -1}{D-1}\leq \frac{2}{1-D}.
\]
so that
\begin{align*} 
|A^{-1}(s,t)| &\leq 
         \Vert A \Vert^{-1} \frac{D^N-1}{D-1}(\alpha^2+(s-t)^2)^{-k} + r^N \Vert A^{-1} \Vert \\
        &\leq \frac{2\Vert A \Vert^{-1}}{1 - D}(\alpha^2+(s-t)^2)^{-k} + r^N \Vert A^{-1} \Vert
   \end{align*}
  Since the above inequality holds for all $N\in\mathbb{N}$ and for any $(s,t)\in\mathbb{Z}^2$, we may find $N\in\mathbb{N}$ so large that $r^N<\|A^{-1} \|^{-1}(\alpha^2+(s-t)^2)^{-k}$ for all $n\geq N$.  Hence for any $(s,t)\in\mathbb{Z}^2$, we have
  \[
  |A^{-1}(s,t)|\leq \left(\dfrac{2\| A^{-1}\|}{1-D}+1  \right)(\alpha^2+(s-t)^2)^{-k}.
  \]
  \end{proof}

 \begin{remark}
 The foregoing result implies, in particular, that $A^{-1}$ is a bounded operator on every $\ell^p(\mathbb{Z}), 1 \leq p \leq \infty$.
 \end{remark}

 \begin{theorem}
 Fix $k\in\mathbb{N}$ and let $X=(x_j : j \in\mathbb{Z})$ be a CIS. Let A = $A_\alpha$ be the bi-infinite matrix whose entries are given by $A(i,j) = (\alpha^2 + (x_i - x_j)^2)^{-k}$, j,k $\in\mathbb{Z}$. Given  $m\in\mathbb{Z}$, let the $m$-th fundamental function be defined as follows:
    $L_{\alpha, m, k}(x) := L_m(x) := \sum_{j\in\mathbb{Z}} A^{-1} (j,m)(\alpha^2 + (x - x_j)^2)^{-k}$.
    For large enough $\alpha$ the following holds:
    \begin{enumerate}
        \item[(i)] The function $L_m$ is continuous throughout $\mathbb{R}$.
        \item[(ii)] Each $L_m$ obeys fundamental interpolatory conditions \\ $L_m (x_j) = \begin{cases}
            1 & j=m\\
            0 & j\neq m
        \end{cases}.$
        \item[(iii)] There exists $p, C > 0$, which depend on $X$ and $k$, such that $|L_m(x)| \leq C(\alpha^2 + |x - x_m|^2)^{-k}$ for every $x \in\mathbb{R}$ and every $m\in\mathbb{Z}$.
        \item[(iv)] If $(b_m)$ satisfies $|b_m| \leq C|m|^{2k-2}$, then the function $x\mapsto \sum_{m} b_mL_m(x)$ is continuous on $\mathbb{R}$.
    \end{enumerate}
    \end{theorem}
    \begin{proof}
    The proof of $(i)$ is a consequence of the Weierstrass $M$-test and decay rate of Theorem $1$.

     For $(ii)$, letting $Id$ denote the identity matrix, we have
     \begin{align*}
            L_m(x_k) &= \sum_{j=-\infty}^\infty A^{-1}(j,m)(\alpha^2+(x_k-x_j)^2)^{-k} \\
            &= \sum_{j=-\infty}^{\infty} A^{-1}(j,m)A(k,j)\\
            &=Id(k,m)\\
            &=\delta_{m,k}.
        \end{align*}
        For $(iii)$, let $m:=m_x$ denote the integer $m$ such that $x\in(x_m,x_{m+1}]$.
        \begin{align*}
            |L_m(x)| &=\sum_{j\in\mathbb{Z}}A^{-1}(j,m)(\alpha^2+(x-x_j)^2)^{-k}\\
            &\leq C_{\alpha,k}(\alpha^2+|x-x_m|^2)^{-k}\left(1+\sum_{j\neq m}A^{-1}(j,m)\left( \dfrac{\alpha^2+(x-x_m)^2}{\alpha^2+(x-x_j)^2}\right)^k \right)\\
            &\leq C_{\alpha,k}(\alpha^2+(x-x_m)^2)^{-k}\left(1+C\sum_{j\neq m}|j-m|^{-2k}\left( \dfrac{\alpha^2+R^2}{\alpha^2+r^2}\right)^k \right)\\
            &\leq C_{\alpha,k,X}(\alpha^2+(x-x_m)^2)^{-k}
        \end{align*}
        Finally, $(iv)$ follows from the decay estimate in Theorem $1$ and the Weierstrass $M$-test.   Indeed, if $|a_j|\leq C|j|^{2k-2}$, then we have
        \begin{align*}
     \left|\sum_{j\in\mathbb{Z}}a_jL_j(x)\right| &\leq C\sum_{j\in\mathbb{Z}}|j|^{2k-2}|L_j(x)| \\
            &\leq C_{\alpha, k} \sum_{j\in\mathbb{Z}}|j|^{2k-2}(\alpha^2+|x-x_j|^2)^{-k} \\
            &\leq C_{\alpha, k,X}.
        \end{align*}
    \end{proof}

    \begin{theorem}
    Let $\alpha > 0$ and $k\in\mathbb{N}$ be fixed. Suppose that $(x_j : j \in \mathbb{Z})$ is a CIS. Let $A = A_{\alpha, k}$ be the bi-infinite matrix whose entries are given by $A(i,j) = (\alpha^2+(x_j-x_i)^2)^{-k}; i,j \in\mathbb{Z}$. Given $p \in [1,\infty]$, and $\Bar{y} := (y_l : l \in\mathbb{Z}) \in \ell^p(\mathbb{Z})$ define \\
    $I[\bar{y}]( x) := \sum_{i\in\mathbb{Z}}(A^{-1}\Bar{y})_i(\alpha^2+(x-x_i)^2)^{-k}, x \in\mathbb{R}$, \\
    where $(A^{-1}\Bar{y})_i$ denotes the $i$th component of the sequence $A^{-1}\Bar{y}$. The following hold:
    \begin{enumerate}
        \item[(i)] The function $ x \mapsto I[\Bar{y}]( x)$ is continuous on $\mathbb{R}$.
        \item[(ii)] If $x$ is any real number, then 
        $I[\Bar{y}](x) = \sum_{i\in\mathbb{Z}}y_iL_i(x)$, 
        where $(L_i : i \in\mathbb{Z})$ is the sequence of fundamental functions introduced in the preceding theorem.
        \item[(iii)] There is a constant $C$, depending on $\alpha, X,$ and $p$ such that 
        $\Vert I[\Bar{y}] \Vert_{L_p\mathbb{(R)}} \leq C \Vert \Bar{y} \Vert_{\ell^p}$, 
        for every $
        \Bar{y} \in \ell^p(\mathbb{Z})$.
    \end{enumerate}
    \end{theorem}
    \begin{proof} 
        $(i)$ The decay bound in Theorem 1 shows that $A^{-1}$ is a bounded operator on $\ell^p(\mathbb{Z})$, hence the sequence $A^{-1}\Bar{y}$ is bounded. Hence the continuity of $I_\alpha[\Bar{y}]$ follows from Theorem 2 $(iv)$.

        $(ii)$ Let $x \in \mathbb{R}$. Then
        \begin{align*}
            I[\bar{y}]( x) =& \sum_{j\in\mathbb{Z}}(A^{-1}\bar{y})_j (\alpha^2 + (x-x_j)^2)^{-k} \\
            =& \sum_{i\in\mathbb{Z}}y_i\left[\sum_{j\in\mathbb{Z}}A^{-1}(j,i) (\alpha^2+(x-x_j)^2)^{-k}\right]\\
            =&\sum_{i\in\mathbb{Z}}y_i L_i(x).
        \end{align*}
        $(iii)$ In light of the Riesz-Thorin interpolation theorem, we need only show this result for $p=1$ and $p=\infty$.  We begin with $p=1$.  Noting $(ii)$, we have
        \begin{align*}
        \int_{\mathbb{R}} |I[\bar{y}](x)|{\rm d}x   =& \sum_{m\in\mathbb{Z}}\int_{x_m}^{x_{m+1}} \left|\sum_{j\in\mathbb{Z}}y_jL_j(x) \right|{\rm d}x \\
        \leq& C_{\alpha,k}R \| \bar{y}\|_{\ell^1} \sum_{m\in\mathbb{Z}}(\alpha^2+\min\{(x_m-x_j)^2,(x_{m+1}-x_j)^2 \})^{-k}\\
        \leq & C_{\alpha,k,X}\| \bar{y}\|_{\ell^1}.
        \end{align*}
        
When $p=\infty$, we have 
\begin{align*}
    \| I[\bar{y}] \|_{L_\infty}\leq &C_{\alpha,k} \| \bar{y} \|_{\ell^\infty}\sup_{x\in\mathbb{R}}\sum_{j\in\mathbb{Z}}(\alpha^2+(x-x_j)^2)^{-k}\\
    \leq & C_{\alpha,k} \| \bar{y} \|_{\ell^\infty}\sum_{j\in\mathbb{Z}}(\alpha^2+r^2(m-j)^2)^{-k}\\
    \leq & C_{\alpha,k,X}\| \bar{y} \|_{\ell^\infty}.
\end{align*}
Now the Riesz-Thorin interpolation theorem yields the result for $1<p<\infty$.
    \end{proof}

\begin{remark}
    Our example has a higher dimension analog, see \cite{Hamm_2017}. It would be interesting to see how these results translate to higher dimensions. The results in this paper give an improved decay estimate for the fundamental function given in \cite{HL_2016}, thus it may be possible to improve the estimate in higher dimension as well.
\end{remark}

\bibliographystyle{plain} 
\bibliography{main}

\begin{thebibliography}{1}

\bibitem{Hamm_2017}
Keaton Hamm.
\newblock Nonuniform sampling and recovery of bandlimited functions in higher dimensions.
\newblock {\em J. Math. Anal. Appl.}, 450(2):1459--1478, 2017.

\bibitem{HL_2016}
Keaton Hamm and Jeff Ledford.
\newblock Cardinal interpolation with general multiquadrics.
\newblock {\em Adv. Comput. Math.}, 42(5):1149--1186, 2016.

\bibitem{HL_2018}
Keaton Hamm and Jeff Ledford.
\newblock Cardinal interpolation with general multiquadrics: convergence rates.
\newblock {\em Adv. Comput. Math.}, 44(4):1205--1233, 2018.

\bibitem{Jaffard_1990}
S.~Jaffard.
\newblock Propri\'et\'es des matrices ``bien localis\'ees'' pr\`es de leur diagonale et quelques applications.
\newblock {\em Ann. Inst. H. Poincar\'e{} C Anal. Non Lin\'eaire}, 7(5):461--476, 1990.

\bibitem{Ledford_2013}
Jeff Ledford.
\newblock Recovery of {P}aley-{W}iener functions using scattered translates of regular interpolators.
\newblock {\em J. Approx. Theory}, 173:1--13, 2013.

\bibitem{Ledford_2015}
Jeff Ledford.
\newblock On the convergence of regular families of cardinal interpolators.
\newblock {\em Adv. Comput. Math.}, 41(2):357--371, 2015.

\bibitem{SS_2009}
Th. Schlumprecht and N.~Sivakumar.
\newblock On the sampling and recovery of bandlimited functions via scattered translates of the {G}aussian.
\newblock {\em J. Approx. Theory}, 159(1):128--153, 2009.

\end{thebibliography}

\noindent {\sc Department of Mathematics \& Computer Science, Longwood University, U.S.A.}

\noindent {\it E-mail address:} {\tt ledfordjp@longwood.edu}

\noindent {\it E-mail address:} {\tt kyle.rutherford@live.longwood.edu}

\end{document}